\documentclass{article}
\usepackage{graphicx}
\usepackage{amsfonts}
\usepackage{amsmath}
\usepackage{url}
\usepackage{amsthm}
\usepackage{amssymb}
\newtheorem {Proposition}{Proposition}[section]
\newtheorem {Lemma}[Proposition] {Lemma}
\newtheorem {Theorem}[Proposition]{Theorem}
\newtheorem {Corollary}[Proposition]{Corollary}
\newtheorem {Remark}[Proposition]{Remark}

\def\N{\mathbb{N}}

\def\R{\mathbb{R}}
 
\usepackage{enumerate}
\usepackage{bbm}
\usepackage[utf8]{inputenc} 
\usepackage[T1]{fontenc}    
\usepackage{hyperref}       
\usepackage{url}            
\usepackage{booktabs}       
\usepackage{amsfonts}       
\usepackage{nicefrac}       
\usepackage{microtype}      
\usepackage{lipsum}
\usepackage{graphicx}
\graphicspath{ {./images/} }
\usepackage{authblk}

\author{Eustasio del Barrio$^{(1)}$\footnote{Research partially supported by FEDER, Spanish Ministerio de Econom\'ia y Competitividad, grant MTM2017-86061-C2-1-P and Junta de Castilla y Le\'on, grants VA005P17 and VA002G18.}, Alberto Gonz\'alez-Sanz$^{(2)}$, and Jean-Michel Loubes $^{(3)}$\footnote{Research partially supported by the AI Interdisciplinary Institute ANITI, which is funded by the French “Investing for the Future – PIA3” program under the Grant agreement ANR-19-PI3A-0004.}\\  $\,$ \\ 
{ $^{(1)(2)}$IMUVA, Universidad de Valladolid, Spain} \\ 
$^{(2)(3)}$IMT, Universit\'e de Toulouse
France\\ $\,$ \\ 
$^{(1)}$tasio@eio.uva.es \quad $^{(2)}$alberto.gonzalez sanz@math.univ-toulouse.fr \\ $^{(3)}$loubes@math.univ-toulouse.fr}

\title{A Central Limit Theorem for Semidiscrete Wasserstein Distances.}

\date{\vspace{-5ex}}

\begin{document}
\maketitle
\begin{abstract}
We address the problem of proving a Central Limit Theorem for the empirical optimal transport cost,  $\sqrt{n}\{\mathcal{T}_c(P_n,Q)-\mathcal{W}_c(P,Q)\}$, in the semi discrete case, i.e when the distribution $P$ is finitely supported. We show that the asymptotic distribution is the supremun of a centered Gaussian process   which is Gaussian under some additional conditions on the probability $Q$ and on the cost. Such results imply the central limit theorem for the $p$-Wassertein distance, for $p\geq 1$.  Finally, the semidiscrete framework provides a control on the second derivative of the dual formulation,  which yields  the first central limit theorem for the optimal transport potentials.
\end{abstract}

\section{Introduction}

A large number of problems in statistics or computer science require the comparison between histograms or,  more generally, measures. Optimal transport has proven  to be an important tool to compare probability measures since it enables to define a metric over the set of distributions which convey their geometric properties., see \cite{verdinelli}. Moreover, together with the convergence of the moments, it metrizes the weak convergence,  see Chapter 7.1. in \cite{villani2003topics}.  It is nowadays used in a large variety of fields, moreover, in statistics and in particular in Machine learning, OT based methods have been developed to tackle problems in fairness as in \cite{ jiang2020wasserstein,pmlr-v97-gordaliza19a,delBarrioGordalizaLoubes2019}, in domain adaptation (\cite{shen2018wasserstein}), or transfer learning (\cite{gayraud2017optimal}). Hence there is a growing need for theoretical results to support such applications and provide theoretical guarantees on the asymptotic distribution.\vspace{5mm} 

In this framework, the optimal transport between  distributions has to be estimated from a random sample of the empirical distributions. Hence some controls have to be added on the rate of convergence of the the empirical version of the optimal transport cost to its true value. Moreover determining the asymptotic behavior of the limit distribution enables to control precisely the deviations of the estimation errors and certify the accuracy of the estimation. \vskip .1in
This work focuses on the semi-discrete optimal transport, i.e. when one of both probabilities is supported on a discrete set.  Such a problem has been studied in many contexts, including on resource allocation problem, points versus demand distribution, positions of sites such that the mean allocation cost is minimal (\cite{Hartmann}), resolution of the incompressible Euler equation  using Lagrangian methods (\cite{Gallou}), non-imaging optics; matching between a point cloud and a triangulated surface; seismic imaging( \cite{MEYRON201913}), generation of blue noise distributions with applications for instance to low-level hardware implementation in printers( \cite{deGoes}), in astronomy (\cite{levy2020fast}). But in a more statistical point of view it can also be  used to implement Goodness-of-fit-tests, in detecting deviations from a density map to have $P\neq Q$, by using the fluctuations of $\mathcal{W}(P_n,Q)$, see \cite{Hartmann} and to the new transport based generalization of the distribution function, proposed by \cite{DELBARRIO2020104671}, when the probability is discrete.
\vskip .1in
The most general formulation of the optimal transport problem considers $\mathcal{X}, \mathcal{Y}$ both Polish spaces. We use the notation $\mathcal{P}(\mathcal{X})$ (resp. $\mathcal{P}(\mathcal{Y})$) for the set of Borel probability measures on $\mathcal{X}$ (resp. $\mathcal{Y}$). The optimal transport problem  between $P\in \mathcal{P}(\mathcal{X})$ and $Q\in \mathcal{P}(\mathcal{Y})$ for the cost $c:\mathcal{X} \times \mathcal{Y}\rightarrow [0,\infty)$ is formulated as the solution of 
\begin{align}\label{kant}
\mathcal{T}_c(P,Q):=\inf_{\gamma \in \Pi(P,Q)}\int_{\mathcal{X}\times \mathcal{Y}} c(\textbf{x},\textbf{y}) d \pi(\textbf{x}, \textbf{y}),
\end{align}
where $\Pi(P,Q)$ is the set of probability measures $\pi \in \mathcal{P}(\mathcal{X}\times \mathcal{Y})$ such that $\pi(A\times \mathcal{Y})=P(A)$ and $\pi(\mathcal{Y} \times B)=Q(B)$ for all $A,B$ measurable sets. \\ \\
If $c$ is continuous and there exist two continuous functions $a\in L^1(P)$ and $b\in L^1(Q)$ such that 
\begin{equation}\label{cost-cond}
    \text{for all $(\mathbf{x},\mathbf{y})\in \operatorname{supp}(P)\times \operatorname{supp}(Q)$,}\ \ c(\mathbf{x},\mathbf{y})\geq a(\mathbf{x})+b(\mathbf{y}),
\end{equation}
then Kantorovich problem \eqref{kant} can be formulated in a dual form, as
\begin{align}\label{dual}
\mathcal{T}_c(P,Q)=\sup_{(f,g)\in \Phi_c(P,Q)}\int f(\textbf{x}) dP(\textbf{x})+\int g(\textbf{y}) dQ(\textbf{y}),
\end{align}
where $\Phi_c(P,Q)=\{ (f,g)\in L_1(P)\times L_1(Q): \ f(\textbf{x})+g(\textbf{x})\leq c(\textbf{x},\textbf{y}) \}$, see for instance Theorem 5.10 in \cite{villani2008optimal}. It is said that  $\psi\in L_1(P)$ is an \emph{optimal transport potential from $P$ to $Q$ for the cost $c$} if there exists $\varphi\in L_1(Q)$ such that the pair $(\psi, \varphi)$ solves \eqref{dual}. 
\vspace{5mm} 

We assume the observation of a sample $X_1, \dots, X_n$ i.i.d. with law $P$. The empirical measure $P_n=\frac{1}{n}\sum_{k=1}^n\delta_{X_k}$ defines a random $\mathcal{T}_c(P_n,Q)$. Supposing that $P,Q\in \mathcal{P}(\R^d)$ and both are absolutely continuous with respect to the Lebesgue measure, \cite{Fournier} proved that $E\mathcal{W}_1(P_n,Q)$ converges to $0$ with an order approximately of $\frac{1}{n^{1/d}}$. By using the triangular inequality, valid in Wassertein distances, we can not expect better order of convergence for the difference $\{\mathcal{W}_1(P_n,Q)-\mathcal{W}_1(P,Q)\}$. But when it comes to the difference $\sqrt{n}\{\mathcal{T}_c(P_n,Q)-E\mathcal{T}_c(P_n,Q)\}$, the results of \cite{delbarrio2021central, delbarrio2019} prove, using the Efron-Stein's inequality, that it follows a Gaussian asymptotic behavior. In the case of $P,Q$ being supported in a finite (resp. countable) set, in \cite{Sommerfeld2018} (resp. \cite{tameling2019}) proved that $\sqrt{n}\{\mathcal{W}_c(P_n,Q)-\mathcal{W}_c(P_n,Q)\}$ has a weak limit $X$. Such a limit is the supremun of a Gaussian process. They start by the identification of the space of distribution supported in a finite set $\mathcal{X}=\{\mathbf{x}_1, \dots, \mathbf{x}_m\}$ with $\R^m$, and then give a proof based on the directional Hadamard differentiability of the functional $(\mathbf{p},\mathbf{q})\mapsto \mathcal{T}_c(\sum_{i=1}^m p_i\delta_{\mathbf{x}_i}, \sum_{i=1}^m{q_i}\delta_{\mathbf{x}_i})$.\vspace{5mm} 

In this work we are concerned with the asymptotic behaviour of $\sqrt{n}\{\mathcal{T}_c(P_n,Q)-\mathcal{T}_c(P_n,Q)\}$ when $P$ is finitely supported but not $Q$. When the quadratic cost is involved and $Q$ is absolutely continuous with respect to the Lebesgue measure, with convex support, \cite{delbarrio2019} proves that the limit is in fact Gaussian. Their approach is also based in some differentiability properties of the optimal transport problem. But one of their main arguments is that the optimal transport potential is unique which is no longer true in general costs and neither in general Polish spaces. We need to use weaker notions of derivative and follow similar arguments of those of \cite{Sommerfeld2018} and \cite{tameling2019} to derive our main result, Theorem~\ref{theo:semi_disc}.  As pointed out before, the result of \cite{Sommerfeld2018} establishes that, if $P$ and $Q$ are both supported in a finite set, then
\begin{align*}
	\sqrt{n}\left(\mathcal{T}_c(P_n,Q)- \mathcal{T}_c(P,Q)\right)\overset{w}\longrightarrow \sup_{\mathbf{z}\in \operatorname{Opt}_c(P,Q)} \mathbb{G}(\mathbf{z}), \ \ \text{where} \ \ \mathbb{G}(\mathbf{z}):=\sum_{i=1}^m z_i  X_i
\end{align*}
and $(X_1,\dots,X_n)$ is a centered Gaussian vector and $\operatorname{Opt}_c(P,Q)$  is the set of solutions of the dual problem \eqref{dual}, both described in section~\ref{sectionCL}.  Lately \cite{tameling2019} extended the same result for probabilities supported in countable spaces.  Our work
covers and generalizes the previously mentioned result of \cite{Sommerfeld2018} through Theorem~\ref{theo:semi_disc}. We use similar tools but we make a slightly different extension. First we define in \eqref{functional} a functional $\mathcal{M}$ countinously differentiable with respect to $\mathbf{p}$ in the positive hyperoctant,  and where the optimization is made in the whole space $\R^m$, see Lemma~\ref{Lemma:dual}. Hence it is a unconstrained optimization problem, consequently we can apply Danskin's theorem (Theorem 4.13 in \cite{Bonnans2000PerturbationAO}) instead of its Lagrangian counterpart, i.e. Theorem 4.24 in \cite{Bonnans2000PerturbationAO}). Such a slightly modification of the functional allows to realize that the central limit theorem holds assuming only finite support on one of both probabilities, which yields the central limit theorem for the semidiscrete case.  Summarizing, the main difference between the approach of \cite{Sommerfeld2018} and  the one we make in section~\ref{sectionCL} is the lack of assumptions on the probability $Q$.  Section~3 goes further and analyses the cases where the transport potential is unique up to additive constants.  In such cases, under some assumption of regularity on the cost and on $Q$,  the limit is not a supremun anymore, but simply a centered Gaussian random variable. The last section studies the semidiscrete O.T. in manifolds and gives a CLT not for the transport cost but for the solutions of the dual problem \eqref{dual}. We underline that if both probabilities are continuous and the space is not one dimensional, we cannot expect such type of central limit for the potentials, since, as commented before,  the expected value of the estimation of the transport cost converges with rate $O(n^{-\frac{1}{d}})$ and no longer $o(n^{-\frac{1}{2}})$. When the two samples are discrete, even if such a rate is $o(n^{-\frac{1}{2}})$, the lack of uniqueness of the dual problem does not allow to prove such type of problems. In consequence, the semidiscrtete is the unique case where such results, for the potentials of the O.T. problem in general dimension, can be expected.
\section{Central Limit Theorem}\label{sectionCL}
The goal is to prove a Central Limit theorem for the optimal transport cost in general Polish spaces $\mathcal{X}, \mathcal{Y}$ for the semi-discrete problem. For that reason during the rest of the paper we will assume that $P$ is supported in a a finite set.  Let $\mathbb{X}=\{\mathbf{x}_1, \dots, \mathbf{x}_m \}\subset \mathcal{X}$ be such that $\mathbf{x}_i\neq \mathbf{x}_j$, for $i\neq j$, and  denote by $\mathcal{P}(\mathbb{X})$ the set of probabilities defined in such set, i.e.  any $P\in \mathcal{P}(\mathbb{X})$ can be written as
\begin{equation}\label{represen}
   \text{$P:=\sum_{k=1}^mp_k\delta_{\mathbf{x}_k}$, where  $p_i>0$, for all $i=1, \dots, m$, and $\sum_{k=1}^mp_k=1$.
 } 
\end{equation}
In consequence $P$ is characterized by the vector $\mathbf{p}=(p_1, \dots, p_m)\in \R^m$. The following result shows that for the previous defined $P$, the optimal transport problem is equivalent to a optimization problem on the first coordinate of the following functional  over $\R^m\times\R^m$ defined as 
\begin{align}\label{functional}
\begin{split}
\mathcal{M}:\R^m\times \R^m &\rightarrow \R \\
    (\mathbf{z},\mathbf{p})&\mapsto \mathcal{M}(\mathbf{z},\mathbf{p}):= \frac{1}{||\mathbf{p}||_1}\sum_{k=1}^{m}|p_k| z_k +\int \min_{i=1, \dots, m} \{ c(\mathbf{y},\mathbf{x}_i ) -z_i \}dQ(\mathbf{y}),
\end{split}
\end{align}
where $||\mathbf{p} ||_1=\sum_{k=1}^{m}|p_k|$. 
The following result, which is a straightforward adaptation of Proposition 4.2 in \cite{delbarrio2019}, shows that for the previous defined $P$, the optimal transport problem is equivalent to a optimization problem on the first coordinate of $\mathcal{M}$. 
\begin{Lemma}\label{Lemma:dual}
Let $P\in \mathcal{P}(\mathbb{X})$ , $Q\in \mathcal{P}(\mathcal{Y})$ and $c$ be a cost satisfying  \eqref{cost-cond}, then 
\begin{align}\label{discrete_dual}
\mathcal{T}_c(P,Q)=\sup_{\mathbf{z}\in \R^m} \mathcal{M}(\mathbf{z}, \mathbf{p}).
\end{align}

\end{Lemma}
\begin{Remark}\label{discrete_dual_remark}
Note that if $\varphi$ denotes an optimal transport potential from $P$ to $Q$ for the cost $c$, then
$$\mathcal{T}_c(P,Q)=\mathcal{M}((\varphi(\mathbf{x}_1),\dots, \varphi(\mathbf{x}_m)), \mathbf{p}),$$
which makes the link between  the optimal transport potentials and optimal values of \eqref{discrete_dual} as $\mathbf{z}=(\varphi(\mathbf{x}_1),\dots, \varphi(\mathbf{x}_m))$.
\end{Remark}
Let $X_1, \dots, X_n$ be a sequence of i.i.d. random variables with law $P$, since $X_k\in \mathbb{X}$ for all $k=1,\dots,n$, the empirical measure $P_n:=\frac{1}{n}\sum_{k=1}^n\delta_{X_k}$ belongs also to $\mathcal{P}(\mathbb{X})$. In consequence it can be written as $P_n:=\sum_{k=1}^mp_k^n\delta_{\mathbf{x}_k}$, where $p_1^n, \dots, p_m^n$ are real random variables such that $p_i^n\geq 0$, for all $i=1, \dots, m$, and $\sum_{k=1}^np_k^n=1$. 
 We want to study the weak limit of the sequence 
$
	\left\lbrace\sqrt{n}\left(\mathcal{T}_c(P_n,Q)- \mathcal{T}_c(P,Q)\right)\right\rbrace_{n\in \N}.
$ 
Consider a centered Gaussian vector,  $(X_1,\dots, X_m)$ with variance matrix
where 
\begin{equation}\label{sigma_matrix}
   \Sigma(\mathbf{p}):= \begin{bmatrix}
p_1(1-p_1) & -p_1p_2  & \cdots& -p_1p_m\\
-p_2p_1 & p_2(1-p_2) & \cdots&-p_2p_m\\
\vdots & \vdots & \ddots&\vdots\\
-p_mp_1 &  \cdots&p_mp_{m-1}&p_m(1-p_m)
\end{bmatrix}.
\end{equation}
and define the class of optimal values
\begin{equation}\label{opt}
    \operatorname{Opt}_c(P,Q):=\left\lbrace  \mathbf{z}\in \R^d: \ \mathcal{T}_c(P,Q)= \mathcal{M}(\mathbf{z}, \mathbf{p})\right\rbrace.
\end{equation}
The following result shows that the class of optimal values $    \operatorname{Opt}_c(P,Q)$ is non-empty and the main theorem makes sense.
\begin{Lemma}\label{Lemma:existence}
Let $P\in \mathcal{P}(\mathbb{X})$, $Q\in \mathcal{P}(\mathcal{Y})$ and $c$ be such that \eqref{cost-cond} holds, then $    \operatorname{Opt}_c(P,Q)\neq\{ \emptyset\}$.
\end{Lemma}
The proof of Lemma~\ref{Lemma:existence} is postponed to the appendix and is based on Lemma~\ref{lem:Properties}, described below and whose proof can be found also in the appendix. The following theorem shows that $
	\left\lbrace\sqrt{n}\left(\mathcal{T}_c(P_n,Q)- \mathcal{T}_c(P,Q)\right)\right\rbrace_{n\in \N}
$ has a weak limit.

\begin{Theorem}\label{theo:semi_disc}
Let $P\in \mathcal{P}(\mathbb{X})$, $Q\in \mathcal{P}(\mathcal{Y})$ and $c$ be such that it satisfies \eqref{cost-cond} and
\begin{align}\label{assumptio:c_disc}
\int c( \mathbf{y}, \mathbf{x}_i)dQ(\mathbf{y})<\infty, \ \text{ for all $i=1,\dots, m$}.
\end{align}
 Then 
\begin{align*}
	\sqrt{n}\left(\mathcal{T}_c(P_n,Q)- \mathcal{T}_c(P,Q)\right)\overset{w}\longrightarrow \sup_{\mathbf{z}\in \operatorname{Opt}_c(P,Q)} \sum_{i=1}^m z_i  X_i,
\end{align*}
where $(X_1,\dots,X_n)\sim \mathcal{N}(\mathbf{0}, \Sigma(\mathbf{p}))$ and $\Sigma(\mathbf{p}) $ is defined in \eqref{sigma_matrix}.
\end{Theorem}
When $\mathcal{X}$ and $\mathcal{Y}$ are contained in the same Polish space $(\mathcal{Z},d)$, a particular cost that satisfies the assumptions of Theorem \ref{theo:semi_disc} is the metric $d$. Then the value${\mathcal{T}_{d^p}(P,Q)}$ follows also the asymptotic behaviour described in Theorem~\ref{theo:semi_disc}.  Using the classic delta-method, the following theorem establishes the asymptotic behavior of the well known $p$-Wasserstsein distance  $\mathcal{W}_p(P,Q)=\sqrt[p]{\mathcal{T}_{d^p}(P,Q)}$.
\begin{Corollary}\label{coro:semi_disc}
Let $P\in \mathcal{P}(\mathbb{X})$ and $Q\in \mathcal{P}(\mathcal{Z}) $ be such that 
\begin{align}\label{assumptio:1_disc}
\int d( \mathbf{y}, \mathbf{x}_0)dQ(\mathbf{y})<\infty, \ \text{ for some $\mathbf{x}_0\in \mathcal{X}$}.
\end{align}
 Then, for any $p\geq 1$, we have
\begin{align*}
	\sqrt{n}\left(\mathcal{W}_p(P_n,Q)- \mathcal{W}_p(P,Q)\right)\overset{w}\longrightarrow \frac{1}{p\left(\mathcal{W}_p\right)^{p-1}}\sup_{\mathbf{z}\in \operatorname{Opt}_{d^p}(P,Q)} \sum_{i=1}^m z_i  X_i,
\end{align*}
where $(X_1,\dots,X_n)\sim \mathcal{N}(\mathbf{0}, \Sigma(\mathbf{p}))$ and $\Sigma(\mathbf{p}) $ is defined in \eqref{sigma_matrix}.
\end{Corollary}
\begin{proof}[Proof of Theorem \ref{theo:semi_disc}]
\mbox{}\\*
We have divided the proof in a sequence of lemmas, finally we will use the Delta-Method for Hadamard differentiable maps. Recall that a function $f:\Theta\rightarrow \R$, defined in an open set $\Theta\subset \R^m$, is said to be \emph{Hadamard directionally differentiable} at $\theta\in \Theta$ if there exists a a function $f'_{\theta}:\R^m\rightarrow \R$ such that 
\begin{equation*}
    \frac{f(\theta+t_n\mathbf{h}_n)-f(\theta)}{t_n}\xrightarrow[n\rightarrow\infty]{} f_{\theta}(\mathbf{h}),\ \ \text{for all sequences $t_n\searrow 0$ and $\mathbf{h}_n\rightarrow \mathbf{h}$.}
\end{equation*}
Let $\Theta\subset\R^m$ be an open set, $\theta\in \Theta$ and $\{X_n\}_{n\in \N}$ be a sequence of random variables such that $X_n:\Omega_n\rightarrow \Theta $ and $r_n(X_n-\theta)\stackrel{w}{\longrightarrow} X$ for some sequence $r_n\rightarrow +\infty$ and some random element $X$ that takes values in $\R^m$. If  $f:\Theta\rightarrow\R$ is Hadamard differentiable at $\theta $ with derivative $f'_{\theta}(\cdot):\R^m\rightarrow \R$, then Theorem~1 in \cite{Romisch}, so-called delta-method, states that  $r_n(f(X_n)-f(\theta))\stackrel{w}{\longrightarrow} f'_{\theta}(X)$. \\
Since $p_i>0$ for all $i=1,\dots,m$, it is enough to consider the hyperoctant
\begin{equation*}
    \mathcal{U}_+:=\{\mathbf{p}\in \R^m:\ p_i>0, \ i=1, \dots, m\}.
\end{equation*}
The aim is to prove that the map \begin{align}\label{Gamma}
\begin{split}
       \Gamma:\mathcal{U}_+&\rightarrow \R\\
    \mathbf{p}&\mapsto \sup_{\mathbf{z}\in \R^m} \mathcal{M}(\mathbf{z}, \mathbf{p}) 
\end{split}
\end{align}
is Hadamard differentiable. With that in mind, we start by proving the following technical lemma. 
\begin{Lemma}\label{lem:Properties}
Let $Q\in \mathcal{P}(\mathcal{Y})$ be such that \eqref{assumptio:c_disc} holds for the cost $c$ which satisfies \eqref{cost-cond}.
Then the functional $\mathcal{M}$ satisfies:
\begin{enumerate}
    \item it is continuous in $\R^m\times \mathcal{U}_+$,
    \item it is continuously differentiable in the second variable over the set $\mathcal{U}_+$, moreover  $$\nabla_{\mathbf{p}} \mathcal{M}(\mathbf{z},\mathbf{p})=\left(\frac{z_i\sum_{j=1}^m p_j-\sum_{j=1}^m z_j p_j}{\left(\sum_{j=1}^m p_j\right)^2} \right)_{i=1,\dots, m}, \ \ \text{for all $\mathbf{z}\in \R^m$,}$$ 
    \item for all $\mathbf{p}_0\in \mathcal{U}_+$ there exists $\alpha\in \R$ and $\epsilon>0$ satisfying that the set

    \begin{equation*}
       \operatorname{lev}_{\alpha}\mathcal{M}(\cdot, \mathbf{p}):= \left\lbrace \mathbf{z}\in \R^m: \mathcal{M}(\mathbf{z}, \mathbf{p})\geq \alpha , \ z_1=0\right\rbrace
    \end{equation*}
        is compact and non empty, for all $\mathbf{p}_0\in \mathcal{U}_+$ such that $|\mathbf{p}_0-\mathbf{p}|<\epsilon$.
\end{enumerate}
\end{Lemma}
Note that Lemma \ref{lem:Properties} yields that the functional $\mathcal{M}$ satisfies the hypothesis of Danskin's theorem, see for instance Theorem 4.13 in \cite{Bonnans2000PerturbationAO}. In consequence $\Gamma$ is directionally differentiable. By proving that $\Gamma$ is locally Lipschitz in $\mathcal{U}_+$, we obtain the following result.
\begin{Lemma}\label{Lemma:Haddmard-dif}
The map $\Gamma$ defined in \eqref{Gamma} is Hadamard differentiable, with derivative
\begin{equation}\label{hada-der}
    \Gamma'_{\mathbf{p}} (\mathbf{q})=\sup_{\mathbf{z}\in \operatorname{Opt}_c(P,Q)}\sum_{i=1}^m \frac{z_i\sum_{j=1}^m p_j-\sum_{j=1}^m z_j p_j}{\left(\sum_{j=1}^m p_j\right)^2}q_i, \ \ \text{for all $\mathbf{q}\in \R^m.$}
\end{equation}
\end{Lemma}
Note that since $P=\sum_{j=1}^m p_j\delta_{\mathbf{x}_j}$ is a probability, $\sum_{j=1}^m p_j=1$ and \eqref{hada-der} can be simplified:
\begin{equation}\label{hada-derP}
    \Gamma'_{\mathbf{p}} (\mathbf{q})=\sum_{i=1}^m \left(z_i-\sum_{j=1}^m z_j p_j \right)q_i, \ \ \text{for all $\mathbf{q}\in \R^m.$}
\end{equation}
Then recall that the multivariate central limit theorem yields that 
\begin{equation*}
    \sqrt{n}\left( \mathbf{p}_n- \mathbf{p}\right)\xrightarrow{w} \mathbf{X}=(X_1,\dots,X_m)\sim N\left(\mathbf{0}, \Sigma(\mathbf{p})\right),
\end{equation*}
where $\Sigma(\mathbf{p})$ is defined in \eqref{sigma_matrix}.
The aforementioned Delta-Method for Hadamard differentiable maps and \eqref{hada-derP} yields that
\begin{equation*}
    \sqrt{n}\left( \Gamma(\mathbf{p}_n)- \Gamma(\mathbf{p})\right)\stackrel{w}\longrightarrow\sup_{\mathbf{z}\in \operatorname{Opt}_c(P,Q)}\sum_{i=1}^m \left(z_i-\sum_{j=1}^m z_j p_j \right) X_i.
\end{equation*}
Since $\mathbf{z}\in \operatorname{Opt}_c(P,Q)$ implies that $\mathbf{z}-\mathbf{1}\sum_{j=1}^m z_j p_j\in \operatorname{Opt}_c(P,Q)$, then we have that
\begin{equation*}
    \sqrt{n}\left( \Gamma(\mathbf{p}_n)- \Gamma(\mathbf{p})\right)\stackrel{w}\longrightarrow\sup_{\mathbf{z}\in \operatorname{Opt}_c(P,Q)}\sum_{i=1}^m z_i  X_i,
\end{equation*}
which concludes the proof.
\end{proof}
\begin{Remark}
Note that, while the idea of proving a central limit theorem for the optimal transport cost through the Hadamard differentiability properties of a suitable functional is not new at all (the same was done by \cite{Sommerfeld2018} and \cite{tameling2019}), the functional used in other references only holds for finitely (or countably) supported $Q\in \mathcal{P}(\mathbb{Y})$, with $\mathbb{Y}=\{ y_1,\dots, y_l\}$.  In \cite{Sommerfeld2018}  the dual problem deals with the functional. More precisely 
\begin{equation*}
 \mathcal{G}(\mathbf{u},\mathbf{v},\mathbf{p}, \mathbf{q}):=\sum_{k=1}^{m}p_k u_k +\sum_{k=1}^{l}q_k v_k .
\end{equation*}
The optimization problem is formulated under the following constraints 
\begin{equation*}
\sup_{\mathbf{u},\mathbf{v}} \mathcal{G}(\mathbf{u},\mathbf{v},\mathbf{p}, \mathbf{q}), \ \ s.t. \ \ u_k+v_j\leq c(x_k,y_j), \ \ \text{for $k\in \{1,\dots, m \} $ and $j\in \{1,\dots, l \} $.}
\end{equation*}
With regards to \cite{tameling2019},  the approach is based on a direct analysis of the primal formulation
which is also a constrained optimization problem.  The definition of $\mathcal{M}$ and Lemma~\ref{discrete_dual} allow to use Danskin's theorem instead of constrained versions such as Theorem 4.24 in \cite{Bonnans2000PerturbationAO}).  Moreover the division by the term $|| \mathbf{p}||_1$ in the definition of $\mathcal{M}$ implies the Hadamard differentiablity not only for a fixed direction, but in a general sense. This is a consequence of the existence of finite solutions of $\sup_{\mathbf{z}\in \R^m} \mathcal{M}(\mathbf{z}, \mathbf{p})$, for every $\mathbf{p}\in \mathcal{U}_+$. 
Moreover, the point~3 in Lemma~\ref{lem:Properties}  is no longer true without such regularization by $|| \mathbf{p}||_1$, and in consequence the argument does not hold. \end{Remark}
\section{Asymptotic Gaussian distribution optimal transport cost.}
The aim of this section is to show that in the case where $\mathcal{X},\mathcal{Y}\subset\R^d$ and under some regularity assumptions, enumerated as (A1)-(A3), the limit distribution is Gaussian.  
Let $Q\in \mathcal{P}(\R^d)$ be a probability measure uniformly continuous with respect to the Lebesgue measure in $\R^d$. Assume that $c(\mathbf{x},\mathbf{y})=h(\mathbf{x}-\mathbf{y})$ where $h:\R^d\rightarrow [0, \infty)$ is a non negative function satisfying: 
\begin{enumerate}
	\item[(A1):] $h$ is strictly convex on $\R^d$.
	\item[(A2):] Given a radius $r\in \R^+$ and an angle $\theta \in (0,\pi) $, there exists some $M:=M(r, \theta)>0$ such that for all $|\mathbf{p} |>M$, one can find a cone 
	\begin{align}
		K(r, \theta, \mathbf{z},\mathbf{p}):=\left\lbrace \mathbf{x}\in \R^d : | \mathbf{x}-\mathbf{p}|| \mathbf{z}|\cos(\theta/2)\leq \left< \mathbf{z},\mathbf{x}-\mathbf{p} \right>\leq  r| \mathbf{z}| \right\rbrace,
	\end{align}
	with vertex at $\mathbf{p}$ on which $h$ attains its maximum at $\mathbf{p}$.
	\item[(A3):] $\lim_{|\mathbf{x} | \rightarrow 0}\frac{h(\mathbf{x})}{|\mathbf{x} | }= \infty $.
\end{enumerate}
Under such assumptions, \cite{gangbo1996} shows the existence of an optimal transport map $T$ solving
\begin{align}\label{monge}
\mathcal{T}_c(P,Q):=\inf_{T}\int c(\textbf{x},T(\textbf{x})) d P(\textbf{x}), \ \ \text{and}\ \ T_{\#}P=Q.
\end{align}
The notation $T_{\#}P$ represents the \emph{Push-Forward} measure, it is the measure such that for each measurable set $A$ we have $T_{\#}P(A):=P(T^{-1}(A))$. The solution of \eqref{monge} is called \emph{optimal transport map} from $P$ to $Q$. Moreover it is defined as the unique Borel function satisfying 
\begin{equation}
   T(\mathbf{x})=\mathbf{x}-\nabla h^*(\nabla \varphi(\mathbf{x})), \ \ \text{where $\varphi$ solves \eqref{dual}.} 
\end{equation}
Here $h^*$ denotes the convex conjugate of $h$, see \cite{rockafellar1970convex}.
Such uniqueness enabled \cite{delbarrio2021central} to deduce the uniqueness, under additive constants, of the solutions of $\varphi$. They assumed (A1)-(A3) to show that if two $c-$concave functions have the same gradient almost everywhere for $\ell_d$ in a connected open set, then both are equal, up to an additive constant. In consequence assuming that $h$ is differentiable, the support of $Q$ is connected and with Lebesgue negligible boundary $\ell_d\left(\partial \operatorname{supp}(Q)\right)=0$, the uniqueness, up to additive constants, of the solutions of \eqref{dual} holds.  
The proof of the main theorem is consequence of Lemma~\ref{Remark_uniq}, which proves that there exists an unique, up to an additive constant, $\mathbf{z}\in \operatorname{Opt}(P,Q) $.  We use within this section the notation $\mathbf{1}:=(1,\dots,1)$.
\begin{Lemma}\label{Remark_uniq}
Under the hypothesis of Theorem~\ref{theo:semi_disc-p}, if $\tilde{\mathbf{z}}$, $\mathbf{z}\in \operatorname{Opt}_c(P,Q)$, then $\tilde{\mathbf{z}}=\mathbf{z}-L\mathbf{1}$ for some constant $L\in \R$.
\end{Lemma}

The following theorem states, under the previous assumptions, that the limit distribution of Theorem~\ref{theo:semi_disc} is the centered Gaussian variable $\sum_{i=1}^m z_i  X_i$, where $\mathbf{z}\in \operatorname{Opt}_c(P,Q)$. 
Note that Lemma~\ref{Remark_uniq} implies that the random variable  $\sum_{i=1}^m z_i  X_i$ follows the same distribution independently of the chosen $\mathbf{z}\in \operatorname{Opt}_c(P,Q)$, which is $\mathcal{N}(0,\sigma^2(P,\mathbf{z})),$  with
\begin{equation}\label{definition_gaussian}
 \sigma^2(P,\mathbf{z})=\operatorname{Var}( \sum_{i=1}^m z_i  X_i)\ \ \text{and } \ \ (X_1,\dots,X_m)\sim\mathcal{N}(\mathbf{0},\Sigma(\mathbf{p})),
\end{equation}
where $\Sigma(\mathbf{p})$ is defined in \eqref{sigma_matrix}.
Since for every $\lambda\in \R$,  we have that $ \sigma^2(P,\mathbf{z})= \sigma^2(P,\mathbf{z}+\lambda\mathbf{1})$,  then the asymptotic variance obtained in the following theorem is well defined.
\begin{Theorem}\label{theo:semi_disc-p}
Let $P\in \mathcal{P}(\mathbb{X})$ and $Q\in \mathcal{P}(\R^d)$ be such that $Q\ll\ell_d$ and its support is connected with Lebesgue negligible boundary. If the cost $c$ satisfies (A1)-(A3) and 
\begin{align*}
\int c( \mathbf{y}, \mathbf{x}_i)dQ(\mathbf{y})<\infty, \ \text{ for all $i=1,\dots, m$.}
\end{align*}
Then 
\begin{align*}
	\sqrt{n}\left(\mathcal{T}_c(P_n,Q)- \mathcal{T}_c(P,Q)\right)\stackrel{w}\longrightarrow \mathcal{N}(0,\sigma^2(P,\mathbf{z})), 
\end{align*}
 with $\sigma^2(P,\mathbf{z})$ defined in \eqref{definition_gaussian}
for $\mathbf{z}\in\operatorname{Opt}_c(P,Q)$.
\end{Theorem}

Since, in particular, the potential costs $c_p=|\cdot|^p$, for $p>0$, satisfy (A1)-(A3), then the following result follows immediately from Theorem~\ref{theo:semi_disc-p} and the Delta-Method for the function $t\mapsto |t|^{\frac{1}{p}}$. Recall that, in the potential cost cases, $\mathcal{T}_p(P,Q)$ denotes the optimal transport cost and $\mathcal{W}_p(P,Q)=\left(\mathcal{T}_p(P,Q)\right)^{\frac{1}{p}}$ the $p$-Wasserstein distance.
\begin{Corollary}\label{Coro:semi_disc-p}
Let $P\in \mathcal{P}(\mathbb{X})$ be as in \eqref{represen} and $Q\in \mathcal{P}(\R^d)$ be such that $Q\ll\ell_d$, has finite moments of order $p$ and its support is connected with Lebesgue negligible boundary.
Then, for every $p>1$, we have that
\begin{align*}
	\sqrt{n}\left(\mathcal{T}_p(P_n,Q)- \mathcal{T}_p(P,Q)\right)\xrightarrow{w} \mathcal{N}(0,\sigma^2(P,\mathbf{z})),
\end{align*}
and 
\begin{align*}
	\sqrt{n}\left(\mathcal{W}_p(P_n,Q)- \mathcal{W}_p(P,Q)\right)\stackrel{w}\longrightarrow \mathcal{N}\left(0,\left(\frac{1}{p\left(\mathcal{W}_p\right)^{p-1}}\right)^2\sigma^2(P,\mathbf{z})\right),
\end{align*}
 with $\sigma^2(P,\mathbf{z})$ defined in \eqref{definition_gaussian}
for $\mathbf{z}\in \operatorname{Opt}_{c_p}(P,Q)$.
\end{Corollary}
\begin{proof}[Proof of Theorem \ref{theo:semi_disc-p}]
\mbox{}\\*
Note that Theorem~\ref{theo:semi_disc} states that \begin{align*}
	\sqrt{n}\left(\mathcal{T}_p(P_n,Q)- \mathcal{T}_p(P,Q)\right)\stackrel{w}\longrightarrow\sup_{z\in \operatorname{Opt}_{c}(P,Q)}\sum_{i=1}^m z_i  X_i,
\end{align*}
where $(X_1,\dots,X_n)\sim \mathcal{N}(\mathbf{0}, \Sigma(\mathbf{p}))$,  $\Sigma(\mathbf{p}) $ is defined in \eqref{sigma_matrix} and
$\operatorname{Opt}_c(P,Q)$ in \eqref{opt}.  Lemma~\ref{Remark_uniq} shows that such class is in fact a singleton, up to additive constants. 
\end{proof}
Note that Corollary~\ref{Coro:semi_disc-p} is a particular case of Corollary~\ref{coro:semi_disc} in the cases where the optimal transport potential is unique---the hypothesis of Theorem~\ref{theo:semi_disc-p} hold---which is the reason why the case $p=1$ can not be considered. With regards to the other potential costs, $p>1$,  it is straightforward to see that the hypothesis (A1)-(A3) hold, see for instance \cite{delbarrio2021central} or  \cite{gangbo1996}. 
\section{A central Limit theorem for the potentials.}
The aim of the section is to deduce a CLT for the potentials $(\psi,\varphi)$ of the transport problem. Recall that it refers to all pairs solving \eqref{dual}. In the semidiscrete case the potentials are pairs formed by $\mathbf{z}=(z_1,\dots,z_m)\in \operatorname{Opt}_c(P,Q)$ and  $
    \label{phi}\varphi(\mathbf{y}):=\inf_{i=1, \dots, m} \{ c(\mathbf{y},\mathbf{x}_i ) -{z_i} \}
$
Note that if $(\psi,\varphi)$ solves \eqref{dual} then $(\psi+C,\varphi-C)$, for all constant $C$. That makes necessary the study the properties of the following functional, defined in  $\langle \mathbf{1}\rangle^{\perp} $ which denotes the orthogonal complement of the vector space generated by $\mathbf{1}=(1,\dots, 1)$,
\begin{align*}
    \mathcal{M}_{\mathbf{p}}:\langle \mathbf{1}\rangle^{\perp} &\longrightarrow \R \\
    \mathbf{z}&\mapsto  \mathcal{M}(\mathbf{z},\mathbf{p}),
\end{align*}
where $\mathcal{M}(\mathbf{z},\mathbf{p})$ is defined in \eqref{functional}.
The idea behind the proof of the main result of the section, Threoem~\ref{Theo:potential}, is to use classical results in $M$-estimation. In consequence, the strictly positiveness of the Hessian of the previously defined $\Gamma$ is required. Such study of the second derivative have already been addressed in \cite{JUN2019}, where they used a completely different notation,  more related to the angle of mathematical analysis than the statistical.  Consequently,  this section changes slightly the notations of the previous ones. In order to make the readers work easier, we will try to be as consistent as possible with the notation adopted in the foregoing sections and, at the same time, to adapt it to the one proposed in \cite{JUN2019}.  First we will assume that $\mathcal{Y}$ is an open domain of a $d$-dimensional Riemannian manifold endowed with the volume measure $\mathcal{V}$ and metric $d$. We will follow the point of view and notations presented in \cite{JUN2019} . For further details we refer to this paper and references therein. \\ \\
Following the approach of \cite{JUN2019}, let's assume  the following assumptions on the cost:
\begin{align}
    \label{Reg}
    \tag{Reg}
    \text{$c(\mathbf{x}_i,\cdot)\in \mathcal{C}^{1,1}(\mathcal{Y})$, for all $i=1,\dots, m$,}
\end{align}
\begin{align}
        \label{Twist}
    \tag{Twist}
    \text{$D_{\mathbf{y}}c(\mathbf{x}_i,\mathbf{y}):\mathcal{Y}\rightarrow T^*_{\mathbf{y}}(\Omega)$ is injective as a function of $\mathbf{y}$, for all $i=1,\dots, m$},
\end{align}
where $D_{\mathbf{y}}c$ denotes the partial derivative of $c$ w.r.t. the second variable.
    For every $i\in \{1, \dots, m\}$ there exists $\Omega_i\subset \R^d$ open and convex set, and a $\mathcal{C}^{1,1}$ diffeomorphism $\exp^c_{i}:\Omega_i\rightarrow \Omega$ such that the functions
\begin{align}
    \label{QC}
    \tag{QC}
        \text{$\Omega_{i}\ni\mathbf{p}\mapsto f_{i,j}(\mathbf{p}):=c(\mathbf{x}_i,\exp^c_{i}\mathbf{p})-c(\mathbf{x}_j,\exp^c_{i}\mathbf{p})$  are quasi-convex for all $j=1,\dots, m$.}
\end{align}
Here quasi-convex, according to \cite{JUN2019}, means that for every $\lambda\in \R^d$ the sets $f_{i,j}^{-1}([-\infty, \lambda])$ are convex. Assumptions on the cost are not enough at all, it becomes necessary that the probability is supported in a \emph{$c$-convex set} $\Omega$, which  means that $ (\exp^c_{i})^{-1}(\Omega)$ is convex, for every $i=1,\dots, m$. Formally,
let $\mathcal{Y}\subset \mathcal{M}$ be a compact $c$-convex set, $P\in \mathcal{P}(\mathbb{X})$ be as in \eqref{represen} and suppose that
  \begin{equation}
 \label{Hol}
     \tag{Hol}
     \text{$Q\in \mathcal{P}(\mathcal{Y})$ satisfies $Q\ll\mathcal{V}$ with density $q\in \mathcal{C}^{o,\alpha}(\mathcal{Y})$}.
 \end{equation}
Recall that $f\in \mathcal{C}^{0,\alpha}(\mathcal{Y})$ if there exist $C>0$ such that $||f(x)-f(y)||\leq Cd(x,y)^{\alpha .} $ The last required assumption in  \cite{JUN2019} is that $Q$ satisfies a \emph{Poincar\'e Wirtinger inequality with constant} $C_{PW}$:  a probability measure $Q$ supported in a compact set $\mathcal{Y}\subset\mathcal{M}$ satisfies a Poincar\'e Wirtinger inequality with constant $C_{PW}$ if for every $f\in \mathcal{C}^1(\mathcal{Y})$ we have that for $Y\sim Q$
 \begin{equation}w
 \label{PW}
     \tag{PW}
     E(|f(Y)-E(f(Y))|)\leq C_{PW} E(|\nabla f(Y)|).
 \end{equation} 
 In order to clarify the feasibility of such assumptions, we will provide some insights on them at the end of the section.\mbox{}\\*
 \cite{JUN2019} proved the following assertions.
\begin{enumerate}
    \item \label{item_Lemma.1} Under assumptions \eqref{Reg} and \eqref{Twist} the function $\mathcal{M}(\cdot, \mathbf{p})$ is concave and differentiable with derivative
 $$\nabla_{\mathbf{z}}\mathcal{M}(\mathbf{z}, \mathbf{p})= (-Q(A_2(\mathbf{z}))+p_2, \dots, -Q(A_m(\mathbf{z}))+p_m ),$$
 where 
\begin{align}\label{A_k}
A_k(\mathbf{z}):=\{ \mathbf{y} \in \R^d :\ \ c(\mathbf{x}_k,\mathbf{y} ) -z_k <c(\mathbf{x}_i,\mathbf{y} )-z_i, \ \ \text{for all $i\neq k$} \}.
\end{align}
\item \label{item_Lemma.3} Under assumptions \eqref{Reg},\eqref{Twist} and \eqref{QC}, the function $\mathcal{M}_{\mathbf{p}}$ is twice continuously differentiable with    Hessian matrix $D_{\mathbf{z}}^2\mathcal{M}(\mathbf{z},{\mathbf{p}})=\left(\frac{\partial^2 \mathcal{M}_{\mathbf{p}}}{\partial z_i\partial z_j}(\mathbf{z})\right)_{i,j=1,\dots,m}$  and partial derivatives 
\begin{align}
\begin{split}\label{Hessian}
 \frac{\partial^2 \mathcal{M}}{\partial z_i\partial z_j}(\mathbf{z},{\mathbf{p}})&=\int_{A_k(\mathbf{z})\cap A_k(\mathbf{z}) } \frac{1}{|\nabla_{\mathbf{y}} c(\mathbf{x}_i,\mathbf{y})-\nabla_{\mathbf{y}} c(\mathbf{x}_j,\mathbf{y})|}dQ(\mathbf{y}), \ \ \text{if  $i\neq j$,}\\
\frac{\partial^2 \mathcal{M}}{\partial^2 z_i}(\mathbf{z},{\mathbf{p}})&=-\sum_{j\neq i}\frac{\partial^2 \mathcal{M}}{\partial z_i\partial z_j}(\mathbf{z},{\mathbf{p}}).  
\end{split}
\end{align}
\item Under assumptions \eqref{Reg},\eqref{Twist} and \eqref{QC}, and if $Q$ satisfies \eqref{PW}, defined below, then there exists a constant $C$ such that 
\begin{equation}\label{estrict_epsilon}
    \text{$\langle D^2_{\mathbf{z}}\mathcal{M}(\mathbf{z},{\mathbf{p}}) \mathbf{v}, \mathbf{v} \rangle \leq -C \epsilon ^3 |\mathbf{v}|^2$, for all $\mathbf{z}\in \mathcal{K}^{\epsilon}$ and $\mathbf{v}\in \langle \mathbf{1}\rangle^{\perp}$,}
\end{equation}
where 
$$ \mathcal{K}^{\epsilon}:=\{\mathbf{z}\in \R^d: \ Q(A_i(\mathbf{z}))>\epsilon, \ \text{ for all $i=1, \dots, m$.} \}.$$
\end{enumerate}

These previous results imply immediately the following Lemma.
\begin{Lemma}\label{Lemma_derivatve2} 
Let $\mathcal{Y}\subset \mathcal{M}$ be a compact $c$-convex set, $P\in \mathcal{P}(\mathbb{X})$ and   $Q\in \mathcal{P}(\mathcal{Y})$ . Under assumptions \eqref{Reg}, \eqref{Twist} and \eqref{QC} on the cost $c$ and \eqref{PW} and \eqref{Hol} on $Q$, we have that the function $\mathcal{M}_{\mathbf{p}}$ is strictly concave and twice continuously differentiable, with 
\begin{align*}
 \nabla \mathcal{M}_p(\mathbf{z})&=\nabla_{\mathbf{z}}\mathcal{M}(\mathbf{z}, \mathbf{p})|_{\langle \mathbf{1}\rangle^{\perp}}, \\
  D^2 \mathcal{M}_p(\mathbf{z})&=D^2_{\mathbf{z}}\mathcal{M}(\mathbf{z}, \mathbf{p})|_{\langle \mathbf{1}\rangle^{\perp}}.
\end{align*}
Moreover if $\bar{\mathbf{z}}\in \langle \mathbf{1}\rangle^{\perp}\cap  \operatorname{Opt}_c(P,Q)$, then there exists a constant $C$ such that 
\begin{equation}\label{Lemma_estrict}
    \text{$\langle D^2\mathcal{M}_{\mathbf{p}}(\bar{\mathbf{z}}) \mathbf{v}, \mathbf{v} \rangle \leq -C \inf_{i=1, \dots, m}|p_i| ^3 |\mathbf{v}|^2$, for all $\mathbf{v}\in \langle \mathbf{1}\rangle^{\perp}$.}
\end{equation}
\end{Lemma}
\begin{proof}
Note that it only remains to prove that \eqref{Lemma_estrict} holds.  But it is a direct consequence of \eqref{estrict_epsilon}. Actually, since $\bar{\mathbf{z}}$ is the unique $\mathbf{z}\in \operatorname{Opt}(P,Q)$ then $Q(A_k(\bar{\mathbf{z}}))=p_k$ for $k=1,\dots,m$ and we can conclude.
\end{proof}
Then we can formulate the main theorem of the section which yields a central limit theorem for the empirical estimation of the potentials. We follow classical arguments of $M$-estimation by writing the function $\mathcal{M}_{\mathbf{p}}(\mathbf{z})$ as $ E(g(X, \mathbf{z}))$ with $X\sim P$ and $g:\mathbb{X}\times \langle \mathbf{1}\rangle^{\perp}\rightarrow \R$ defined by 
\begin{align}\label{g}
        g(\mathbf{x}_k, \mathbf{z})&:= z_k+\int \inf_{i=1, \dots, m} \{ c(\mathbf{y},\mathbf{x}_i ) -z_i \}dQ(\mathbf{y}),
\end{align}
for each $\mathbf{z}=(z_1,\dots,z_m)$. The weak limit is a centered multivariate Gaussian distribution with covariate matrix, depending on the optimal $\tilde{\mathbf{z}}$, defined by the bilinear map 
\begin{equation}\label{eq:sigmaZ}
\Sigma(\tilde{\mathbf{z}})=
\left(D^2\mathcal{M}_{\mathbf{p}}(\tilde{\mathbf{z}} )\right)^{-1}A
    \left(D^2\mathcal{M}_{\mathbf{p}}(\tilde{\mathbf{z}} )\right)^{-1},\ \text{
where $A=\sum_{i=1}^m p_i \nabla_{\tilde{\mathbf{z}}}g(\mathbf{x}_i, \mathbf{z}) \nabla_{\tilde{\mathbf{z}}}g(\mathbf{x}_i, \mathbf{z})^{t}.$}
\end{equation}
\begin{Theorem}\label{Theo:potential}
Let $\mathcal{Y}\subset \mathcal{M}$ be a compact $c$-convex set, $P\in \mathcal{P}(\mathbb{X})$ and $Q\in \mathcal{P}(\mathcal{Y})$. Under assumptions \eqref{Reg}, \eqref{Twist} and \eqref{QC} on the cost $c$, and \eqref{PW} and \eqref{Hol} on $Q$, we have that
\begin{equation}
    \sqrt{n}\left(\hat{\mathbf{z}}_n-\tilde{\mathbf{z}}\right)\stackrel{w}\longrightarrow N(\mathbf{0},\Sigma(\tilde{\mathbf{z}})),
\end{equation}
where $\tilde{\mathbf{z}}\in \langle \mathbf{1}\rangle^{\perp}\cap \operatorname{Opt}_c(P,Q)$ (resp. $\hat{\mathbf{z}}_n\in \langle \mathbf{1}\rangle^{\perp}\cap \operatorname{Opt}_c(P_n,Q)$), and 
$\Sigma(\tilde{\mathbf{z}})$ is defined in \eqref{eq:sigmaZ}.
\end{Theorem}
\begin{proof}
Now let $g:\mathbb{X}\times \langle \mathbf{1}\rangle^{\perp}\rightarrow \R$ be defined in \eqref{g}. It satisfies that if $X\sim P$ then $ E(g(X, \mathbf{z}))=\mathcal{M}_{\mathbf{p}}(\mathbf{z})$.
Lemma~\ref{Lemma_derivatve2} implies in particular that:
\begin{enumerate}[(i)]
    \item the function $\mathbf{z}\mapsto g(\mathbf{x}_k, \mathbf{z})$ is concave for every $\mathbf{x}_k$ .
    \item There exists a unique $$\tilde{\mathbf{z}}\in \arg\sup_{
    \mathbf{z}\in \langle \mathbf{1}\rangle^{\perp}} E(g(X, \mathbf{z})).$$
    \item The empirical potential is defined as $$\hat{\mathbf{z}}_n\in \arg\sup_{
    \mathbf{z}\in \langle \mathbf{1}\rangle^{\perp}} P_n(g(X, \mathbf{z})).$$
     \item The function $\mathbf{z}\mapsto E(g(X, \mathbf{z}))$ is twice differentiable in $\tilde{\mathbf{z}}$ with strictly negative definite Hessian matrix  $D^2\mathcal{M}_{\mathbf{p}}(\tilde{\mathbf{z}} )$.
     \item For every $\mathbf{z}\in \R^{m-1}$ we have that  $$|\nabla_{\mathbf{z}}g(\mathbf{x}_k, \mathbf{z})|\leq |(Q(A_2(\mathbf{z}))+1, \dots, Q(A_m(\mathbf{z}))+1 )|\leq \sqrt{2(m-1)}.$$
     
\end{enumerate}
Then all the assumptions of Corollary 2.2 in \cite{Huang2007CENTRALLT} are satisfied by the function $g$. In consequence we have the limit
\begin{equation}
    \sqrt{n}\left(\hat{\mathbf{z}}_n-\tilde{\mathbf{z}}\right)\xrightarrow[]{w} N\left(\mathbf{0},\left(D^2\mathcal{M}_{\mathbf{p}}(\tilde{\mathbf{z}} )\right)^{-1}A
    \left(D^2\mathcal{M}_{\mathbf{p}}(\tilde{\mathbf{z}} )\right)^{-1}\right),
\end{equation}
where $A=E\left( \nabla_{\mathbf{z}}g(X, \mathbf{z}) \nabla_{\mathbf{z}}g(X, \mathbf{z})^{t}\right).$ Note that computing $A$ we obtain the expression  \eqref{eq:sigmaZ}.
\end{proof}
For $\tilde{\mathbf{z}}$ defined as in Theorem~\ref{Theo:potential}, set
\begin{equation}\label{c-conj}
   \varphi(\mathbf{y}):=\inf_{i=1, \dots, m} \{ c(\mathbf{y},\mathbf{x}_i ) -\bar{z_i} \}
\end{equation}
and note that it is an optimal transport map from $Q$ to $P$, set also the value $i(y)\in \{ 1,\dots, m\}$ where the infumum of \eqref{c-conj} is attained. As previously  we can define their empirical counterparts 
\begin{equation}\label{c-con_n}
   \varphi_n(\mathbf{y}):=\inf_{i=1, \dots, m} \{ c(\mathbf{y},\mathbf{x}_i ) -\hat{z}_i \},
\end{equation}
which is an optimal transport map from $Q$ to $P_n$, and $i_{n}(y)$ the index where the infimum of \eqref{c-con_n} is attained. Then we have that 
\begin{align}\label{sandwich}
\sqrt{n} (\hat{z}_{i_n(y)}-\bar{z}_{i_n(y)}) \leq \sqrt{n}(\varphi_n(\mathbf{y})-\varphi(\mathbf{y}))\leq  \sqrt{n} (\hat{z}_{i(y)}-\bar{z}_{i(y)}).
\end{align}
We can take supremums over $\mathbf{y}$ in both sides of \eqref{sandwich} and derive that
\begin{align*}
  \sqrt{n}\sup_{i=1, \dots, m}( \hat{z}_{i}-\bar{z}_{i})=\sqrt{n}\sup_{\mathbf{y}\in \mathcal{Y}}(\varphi_n(\mathbf{y})-\varphi(\mathbf{y})).
\end{align*}
By symmetry we have that 
\begin{align*}
  \sqrt{n}\sup_{i=1, \dots, m}| \hat{z}_{i}-\bar{z}_{i}|=\sqrt{n}\sup_{\mathbf{y}\in \mathcal{Y}}|\varphi_n(\mathbf{y})-\varphi(\mathbf{y})|,
\end{align*}
which implies the following corollary.
\begin{Corollary}\label{Corollary_pot}
Under the hypothesis of Theorem~\ref{Theo:potential}, for $\varphi$ and $\varphi_n$ defined in \eqref{phi} and \eqref{phin}, we have that 
$$
\sqrt{n}\sup_{\mathbf{y}\in \mathcal{Y}}|\varphi_n(\mathbf{y})-\varphi(\mathbf{y})|\stackrel{w}\longrightarrow \sup_{i=1,\dots,m} |N_i|,
$$
where $(N_1,\dots, N_m)\sim N(\mathbf{0},\Sigma(\tilde{\mathbf{z}})).$
\end{Corollary}
We will conclude by  some comments on the assumptions made in this section.
\begin{enumerate}
    \item Note that if we consider $\mathcal{M}=\R^d$ and the quadratic cost, then \eqref{Reg}, \eqref{Twist} and \eqref{QC} are obviously satisfied, by taking the function $\exp_j$ as the identity. Actually the map  $\mathbf{y}\mapsto|\mathbf{x}_j-\mathbf{y} |^2$ is $\mathcal{C}^{\infty}(\R^d)$ and $\mathbf{y}-\mathbf{x}_j$ is its derivative w.r.t. $\mathbf{y}$. Finally note that the function
    $$ \R^d\ni\mathbf{p}\mapsto |\mathbf{x}_i-\mathbf{p}|^2-|\mathbf{x}_j-\mathbf{p}|^2=|\mathbf{x}_i|^2-|\mathbf{x}_j|^2+\langle\mathbf{x}_j-\mathbf{x}_i,\mathbf{p}\rangle$$
    is linear in $\mathbf{p}$ and consequently quasi-convex.
\item Assumption~\eqref{PW} on the probability $Q$ has been widely studied in the literature for its implications in PDEs, see \cite{Acosta}. They proved that \eqref{PW} holds for a uniform distribution on a convex set $\Omega$. In \cite{Rathmair}, Lemma 1 claims that \eqref{PW} is equivalent to the bound of $\inf_{t\in \R} E(|f(Y)-t|)$, for every $f\in \mathcal{C}^1(\mathcal{Y})$. 
Let $Y\sim Q$ be such that there exists a $\mathcal{C}^1(\mathcal{Y})$ map $T$ satisfying the relation $T(U)=Y$, where $U$ follows a uniform distribution on a compact convex set $A$. Since $f\circ T\in \mathcal{C}^1(A)$,   by the powerful result of \cite{Acosta}, there exists $C_A>0$ such that 
\begin{align*}
       \inf_{t\in \R} E(|f(Y)-t|)&= \inf_{t\in \R} E(|f(T(U))-t|)\leq C_{A} E(|\nabla f(T(U))|\cdot ||T'(U)||_2)\\
    &\leq C_{A}\sup_{\mathbf{u}\in A}||T'(\mathbf{u})||_2E(|\nabla f(T(U))|),
\end{align*}
where $|| T'(U)||_2$ denotes the matrix norm. We deduce that in such cases \eqref{PW} holds. Note that the existence of such map relies on the well known existence of continuously differentiable optimal transport maps, which is treated by Caffarelli's theory. We refer to the most recent work \cite{cordero2019regularity} and references therein. However, as pointed out in \cite{JUN2019},  more general probabilities can satisfy that assumption such as radial functions on $\R^d$ with density 
\begin{equation*}
    \frac{p(|\mathbf{x}|)}{| \mathbf{x}|^{d-1}}, \ \ \text{for $|\mathbf{x}|\leq R$,  with $p=0$ in $[0,r] $ and concave in $[r,R]$. }
\end{equation*}
 Moreover the spherical uniform $\mathbb{U}_d$, used in \cite{delbarrio2020centeroutward} to generalize the distribution function to general dimension, where we first choose the radium uniformly and then, independently, we choose a point in the sphere $\mathbf{S}_{d-1}$, also satisfies \eqref{PW}. This can be proved by using previous reasoning with the function $T(\mathbf{x})=\mathbf{x}| \mathbf{x}|^{d-1}$, which is continuously differentiable. But note that such probability measure does not satisfy the regularity condition of continuous density over a convex set, in consequence some additional work should be done which is left as a future work. In the same way as the regularity of the transport can be derived in the continuous case by a careful treatment of the Monge-Amp\'ere equation, see \cite{DELBARRIO2020104671}, we conjecture that Theorem~\ref{Theo:potential} could hold also in that discrete framework. 
\end{enumerate}

\section{Appendix}
\begin{proof}[Proof of Lemma \ref{Lemma:existence}]
\mbox{}\\*
Note that Lemma~\ref{lem:Properties} yields that there exist $\alpha>0$ such that  $$\operatorname{lev}_{\alpha}\mathcal{M}(\cdot, \mathbf{p}):= \left\lbrace \mathbf{z}\in \R^m: \mathcal{M}(\mathbf{z}, \mathbf{p})\geq \alpha , \ z_1=0\right\rbrace$$ is non empty and compact. Lemma \ref{lem:Properties} also claims that $\mathcal{M}$ is continuous, then 
$$\sup_{\mathbf{z}\in \R^m} \mathcal{M}(\mathbf{z}, \mathbf{p})=\sup_{\mathbf{z}\in \operatorname{lev}_{\alpha}\mathcal{M}(\cdot, \mathbf{p})} \mathcal{M}(\mathbf{z}, \mathbf{p}) $$
whose supremun is attained in some $\mathbf{z}\in \operatorname{lev}_{\alpha}\mathcal{M}(\cdot, \mathbf{p})$ by Weierstrass extreme value theorem.
\end{proof}
\begin{proof}[Proof of Lemma \ref{lem:Properties}]
\mbox{}\\*
To prove the first point we realize that  $\mathcal{M}=F_1+F_2$, where
$$ F_1((\mathbf{z},\mathbf{p})):=\int \inf_{i=1, \dots, m} \{ c(\mathbf{y}-\mathbf{x}_i ) -z_i \}dQ(\mathbf{y})\ \ \text{and} \ \ F_2((\mathbf{z},\mathbf{p})):= \frac{1}{||\mathbf{p}||_1}\sum_{k=1}^{m}|p_k| z_k $$
are both continuous functions in $\R^m\times \mathcal{U}_+$. An easy computation shows the point 2. of Lemma~\ref{lem:Properties}. For the part 3,  we set $\mathbf{p}_0\in \mathcal{U}_+$ and a sequence $\{\mathbf{p}_n\}_{n\in \N}$ converging to $\mathbf{p}_0$ then the measure 
$ P_n:=\frac{1}{\sum_{j=1}^m p_j^n}\sum_{j=1}^m p_i^n\delta_{\mathbf{x}_j}$ converges weakly to $ P_0=\frac{1}{\sum_{j=1}^m p_j^0}\sum_{j=1}^m p_i^0\delta_{\mathbf{x}_j}$. In consequence $\mathcal{T}_c(P_n,Q)\rightarrow \mathcal{T}_c(P,Q)=\lambda>0$, when $n\rightarrow\infty$. This means that there exists some $\epsilon>0$ such that $\sup_{\mathbf{z}\in \R^m}\mathcal{M}(\mathbf{z}, \mathbf{p})\geq \alpha:=\lambda/2$, for all $\mathbf{p}$ such that $|\mathbf{p}_0-\mathbf{p}|<\epsilon$. Since $\mathcal{U}_+$ is open, we can assume that $\mathbb{B}_{2\epsilon}(\mathbf{p}_0)=\{ \mathbf{q}\in \R^m: |\mathbf{p}_0-\mathbf{q}|<2\epsilon\ \}\subset \mathcal{U}_+$. This means that
\begin{equation}\label{compactes}
   1-\epsilon>\frac{p_j}{||\mathbf{p}||_1}>\epsilon , \ \ \text{for all $\mathbf{p}\in \mathbb{B}_{\epsilon}(\mathbf{p}_0)$}
\end{equation}
Therefore the set
$
       \operatorname{lev}_{\alpha}\mathcal{M}(\cdot, \mathbf{p}):= \left\lbrace \mathbf{z}\in \R^m: \mathcal{M}(\mathbf{z}, \mathbf{p})\geq \alpha , \ z_1=0\right\rbrace
$
is not empty for every $\mathbf{p}\in \mathbb{B}_{\epsilon}(\mathbf{p}_0)$. The only point remaining concerns the compactness of $\operatorname{lev}_{\alpha}\mathcal{M}(\cdot, \mathbf{p})$. Note that 
\begin{align*}
 &\frac{1}{||\mathbf{p}||_1}\sum_{k=1}^{m}|p_k| z_k +\int \inf_{i=1, \dots, m} \{ c(\mathbf{y}-\mathbf{x}_i ) -z_i \}dQ(\mathbf{y})\\
 &\leq \frac{1}{||\mathbf{p}||_1}\sum_{k=1}^{m}p_k z_k +\inf_{i=1, \dots, m} \{ \int c(\mathbf{y}-\mathbf{x}_i )dQ(\mathbf{y}) -z_i \}\\
 &\leq \frac{1}{||\mathbf{p}||_1}\sum_{k=1}^{m}p_k z_k +\sup_{i=1, \dots, m}\int c(\mathbf{y}-\mathbf{x}_i )dQ(\mathbf{y}) -\sup_{i=1, \dots, m} z_i\\
 &=\frac{1}{||\mathbf{p}||_1}\sum_{k=1}^{m}p_k \left(z_k-\sup_{i=1, \dots, m} z_i\right) +\sup_{i=1, \dots, m}\int c(\mathbf{y}-\mathbf{x}_i )dQ(\mathbf{y}) \\
  &=\frac{1}{||\mathbf{p}||_1}\sum_{k=2}^{m}p_k \left(z_k-\sup_{i=1, \dots, m} z_i\right) +\sup_{i=1, \dots, m}\int c(\mathbf{y}-\mathbf{x}_i )dQ(\mathbf{y}) +\frac{1}{||\mathbf{p}||_1}p_1 \left(-\sup_{i=1, \dots, m} z_i\right)\\
  &\leq K -\frac{p_1}{||\mathbf{p}||_1} \left(\sup_{i=1, \dots, m} z_i\right),
\end{align*}
where $K=\sup_{i=1, \dots, m}\int c(\mathbf{y}-\mathbf{x}_i )dQ(\mathbf{y})\in \R$.
From \eqref{compactes} we deduce that $\sup_{k} z_k $ has to be bounded $\sup_{k} z_k \leq M$. The task now becomes to show that $\inf_{k} z_k $ has to be bounded. Since $z_1=0$ then 
\begin{align*}
 \mathcal{M}(\mathbf{z},\mathbf{p} )
 &\leq \inf_j \frac{p_j}{||\mathbf{p}||_1} \inf_{k} z_k+\sup_j \frac{p_j}{||\mathbf{p}||_1} \sup_{k} z_k+ \sup_{i}\int c(\mathbf{y}-\mathbf{x}_i )dQ(\mathbf{y})-z_1\\
 &\leq M+K+\inf_j \frac{p_j}{||\mathbf{p}||_1}\inf_{k} z_k.
\end{align*}
Finally \eqref{compactes} implies that $\inf_{k} z_k \geq M'$. In consequence the set $\operatorname{lev}_{\alpha}\mathcal{M}(\cdot, \mathbf{p})$ is bounded, it is clear that also closed, hence compact.
\end{proof}
\begin{proof}[Proof of Lemma \ref{Lemma:Haddmard-dif}]
\mbox{}\\*
Danskin's theorem (Theorem 4.13 in \cite{Bonnans2000PerturbationAO}) considers an optimization problem 
\begin{equation*}
    G(u)=\max_{x\in X}g(x,u)\ \ s.t. \ \ x\in \Phi,
\end{equation*}
where $u\in U$ and it is a Banach space with the norm $|\cdot|$, $X$ is a topological Hausdorff space, $\Phi\subset X$ is nonempty and closed and $g$ is continuous. If the function $g(x,\cdot)$  is differentiable, $g$ and its derivative w.r.t. $u$ in the direction $u$ $D_u g(x,u)$ are countinous in $X\times U$
and for all $u_0\in U$ there exists $\alpha\in \R$ and $\epsilon>0$ satisfying that the set
    \begin{equation*}
       \operatorname{lev}_{\alpha}\mathcal{M}(\cdot, \mathbf{p}):= \left\lbrace x\in \Phi: f(x,u)\geq \alpha \ \right\rbrace
    \end{equation*}
        is compact and non empty, for all $u$ such that $|u-u_0|<\epsilon$.
 Then the function $G$ is Fr\'echet directionally differentiable in $u_0$, which means that, for every $h\in U $, the limit $$G'_{u_0}(h)= \lim_{t\downarrow 0}\frac{G(u_0+th)-G(u_0)}{t} $$
exists and satisfies that $$G(u_0+h)=G(u_0)+G'_{u_0}(h)+o(|h|). $$
 Note that Lemma \ref{lem:Properties} yields that the functional $\mathcal{M}$ satisfies the assumptions of Danskin's theorem in $\R^d\times \mathcal{U}_+$, but $\mathcal{U}_+$ is not a Banach space, consequently some additional work has to be done.  Set $\mathbf{p}\in \mathcal{U}_+$ and note that there exists some $\epsilon>0$ such that $\overline{\mathbb{B}_{\epsilon}(\mathbf{p})}\subset \mathcal{U}_+$.  The function  $\mathcal{M}$ restricted to such set  $\mathcal{M}|_{\R^m\times\overline{\mathbb{B}_{\epsilon}(\mathbf{p})}}$ is a continuous  function in $\R^m\times\overline{\mathbb{B}_{\epsilon}(\mathbf{p})}$, then we can apply Theorem~1 in \cite{Whitney} to conclude that there exists a function $\tilde{\mathcal{M}}:\R^m\times \R^m\rightarrow \R$ such that 
 \begin{itemize}
 \item $\tilde{\mathcal{M}}=\mathcal{M}$ in $\R^m\times\overline{\mathbb{B}_{\epsilon}(\mathbf{p})}$,
 \item $\tilde{\mathcal{M}}$ is analytic in $\R^m\times \R^m\setminus\left(\R^m\times\overline{\mathbb{B}_{\epsilon}(\mathbf{p})}\right)$.
 \end{itemize}
 Note that the equality
 $$\Gamma(\mathbf{p})= \sup_{\mathbf{z}\in \R^m} \mathcal{M}(\mathbf{z}=\sup_{\mathbf{z}\in \R^m}\tilde{ \mathcal{M}}(\mathbf{z},\mathbf{p})$$
 implies that we can apply Danskin's theorem  to $\tilde{\mathcal{M}}$ and   in consequence $\Gamma$ is  Fr\'echet directionally differentiable in $\mathbf{p}.$ We claim that $\Gamma$ is locally Lipschitz. To prove it recall that Lemma~\ref{Lemma:existence} yields that there exist $\mathbf{z}$  such that $\mathcal{M}(\mathbf{z},\mathbf{p})=\mathcal{T}_c(P,Q)$ and $z_1=0$.
Lemma \ref{lem:Properties} claims also that there exists $\alpha\in \R$ and $\epsilon>0$ such $\operatorname{lev}_{\alpha}\mathcal{M}(\cdot, \mathbf{p})$ is non empty and compact for all $\mathbf{p}'\in \mathbb{B}_{\epsilon}(\mathbf{p})$. Then for all $\mathbf{p}'\in \mathbb{B}_{\epsilon}(\mathbf{p})$ and all $\mathbf{z}'$ solving  $\mathcal{M}(\mathbf{z}',\mathbf{p}')=\mathcal{T}_c(P',Q)$ and $z'_1=0$, we have that $|z'_1|\leq M$. Moreover such $\epsilon$ can be chosen  small enough in order to have $\mathbb{B}_{2\epsilon}(\mathbf{p})\subset \mathcal{U}_+$
In consequence we have that
\begin{align*}
   \Gamma(\mathbf{p})-\Gamma(\mathbf{p}')
   &\leq \mathcal{M}(\mathbf{z},\mathbf{p})-\mathcal{M}(\mathbf{z},\mathbf{p}')=\sum_{j=1}^m \left(\frac{p_k}{|\mathbf{p}|_1}-\frac{p'_k}{|\mathbf{p}'|_1}\right)z_j\\
   &\leq M \sum_{j=1}^m \left|\frac{p_k}{|\mathbf{p}|_1}-\frac{p'_k}{|\mathbf{p}'|_1}\right|=M|F(\mathbf{p})-F(\mathbf{p}')|_1, 
\end{align*}
 where $F(\mathbf{p}):=\left( \frac{p_1}{|\mathbf{p}|_1}, \dots, \frac{p_m}{|\mathbf{p}|_1} \right)$. We remark that the reverse inequality holds exactly with the same arguments. Then it suffices to prove that $F$ is Lipchitz in $\mathbb{B}_{\epsilon}(\mathbf{p})$. It can be seen by computing for any $\mathbf{p}'\in \mathbb{B}_{\epsilon}(\mathbf{p})$ the derivative
 \begin{equation*}
    F'(\mathbf{p}')=\frac{1}{\left(\sum_{i= 1}^m p'_i\right)^2} \begin{bmatrix}
{-\sum_{i\neq 1}p'_i} & -1  & \cdots& -1\\
-1 & {-\sum_{i\neq 2}p'_i} & \cdots&-1\\
\vdots & \vdots & \ddots&\vdots\\
-1 &  \cdots& -1&{-\sum_{i\neq m}p'_i}\\
\end{bmatrix},
 \end{equation*}
 and checking that the matrix norm  $||F'(\mathbf{p}') ||_{1}\leq K:= \frac{m-1+|\mathbf{p}|_1+\epsilon }{\left(|\mathbf{p}|_1-\epsilon \right)^2}$. Note that this inequality holds since $\epsilon$ is small enough to make $|\mathbf{p}|_1-\epsilon>0 $. Therefore we have that 
  \begin{equation*}
|\Gamma(\mathbf{p})-\Gamma(\mathbf{p}')|\leq MK| \mathbf{p}-\mathbf{p}'|_1\leq \sqrt{m}MK| \mathbf{p}-\mathbf{p}'|, \ \ \text{for all $\mathbf{p}'\in \mathbb{B}_{\epsilon}(\mathbf{p})$}.
 \end{equation*}
 Then $\Gamma$ is locally Lipschitz and directionally differentiable in $\mathcal{U}_+$. By Proposition 2.49 in \cite{Bonnans2000PerturbationAO},  we conclude that it is differentiable at $\mathbf{p}$ in the Hadamard sense.   
\end{proof}
\begin{proof}[Proof of Lemma \ref{Remark_uniq}]
\mbox{}\\*
Set $\tilde{\mathbf{z}},\mathbf{z}\in \operatorname{Opt}(P,Q)$. Lemma~\ref{Lemma:dual} implies that the $c$-concave functions $\psi( \mathbf{y}):=\inf_{i=1, \dots, m} \{ c(\mathbf{x}_i, \mathbf{y}) -z_i \}$  and $\tilde{\psi}( \mathbf{y}):=\inf_{i=1, \dots, m} \{ c(\mathbf{x}_i, \mathbf{y}) -\tilde{z}_i \}$ are both solutions of the dual optimal transport problem from $Q$ to $P$. Moreover, from optimal transport theory,  the map ${\psi}^c$ (resp. $\tilde{\psi}^c$) satisfies that $\varphi(\mathbf{x}_j)=z_j$ (resp. $\tilde{\varphi}(\mathbf{x}_j)=\tilde{z}_j$). Hence  Theorem~2.6 in \cite{delbarrio2021central} implies that $\psi=\tilde{\psi}+L$ holds in the interior of the support of $Q$. Since
$$\tilde{z}_i:=\tilde{\psi}^c (\mathbf{x}_i)=\psi^c(\mathbf{x}_i)-L= z_i -L,$$
we have that $\tilde{\mathbf{z}}=\mathbf{z}-L$ for some constant $L\in \R$.
\end{proof}
\bibliographystyle{apalike}  
\bibliography{references}  


\end{document}